\documentclass[letterpaper,10 pt,conference]{ieeeconf}  
\IEEEoverridecommandlockouts                              

\usepackage{comment}

\usepackage{subfigure}

\usepackage{xcolor}
\usepackage{xparse}
\usepackage{mathrsfs}
\usepackage{graphics} 
\usepackage{amsmath} 
\usepackage{amssymb}  
\usepackage{amsthm}
\usepackage{enumitem} 
\usepackage{hyperref}%
\hypersetup{colorlinks=true,  linkcolor=black,  breaklinks=true,  urlcolor=black,  citecolor=black}

\usepackage{soul}
\usepackage{url}

\usepackage[prependcaption, colorinlistoftodos]{todonotes}

\DeclareMathOperator*{\rank}{rank}

\newcommand\oprocendsymbol{\hbox{$\triangle$}}
\newcommand\oprocend{\relax\ifmmode\else\unskip\hfill\fi\oprocendsymbol}

\DeclareSymbolFont{bbold}{U}{bbold}{m}{n}
\DeclareSymbolFontAlphabet{\mathbbold}{bbold}

\newtheorem{theorem}{Theorem}
\newtheorem{remark}[theorem]{Remark}
\newtheorem{example}[theorem]{Example}
\newtheorem{lemma}[theorem]{Lemma}
\newtheorem{definition}[theorem]{Definition}
\newtheorem{proposition}[theorem]{Proposition}
\newtheorem{corollary}[theorem]{Corollary}

\newcommand {\be}{\begin{equation}}
\newcommand {\ee}{\end{equation}}

\newcommand{\im}{{\rm Im}}

\title{\LARGE\bf  On the equivalence of model-based and data-driven approaches to the design of unknown-input observers}

\author{Giorgia Disar\`o and Maria Elena Valcher
\thanks{G. Disar\`o and  M.E. Valcher are with the Dipartimento di Ingegneria dell'Informazione,
 Universit\`a di Padova,
    via Gradenigo 6B, 35131 Padova, Italy, e-mail:  \texttt{giorgia.disaro@phd.unipd.it, meme@dei.unipd.it}}
   }
  \date{}

\begin{document}
\maketitle

\begin{abstract}
In this paper we investigate a data-driven approach to the design of an unknown-input observer (UIO).
Specifically, we provide necessary and sufficient conditions for the existence  of an unknown-input observer for a discrete-time linear time-invariant (LTI) system, designed based only on some available data, obtained on a finite time window. 
We  also prove that, under weak assumptions on the collected data, the solvability conditions derived by means of the data-driven approach are in fact equivalent to those obtained through the model-based one. 
In other words, 
the data-driven conditions 
do not impose further constraints with respect to the classic 
 model-based ones, expressed in terms of the original system matrices.
\end{abstract}

\section{Introduction}
 
In many control engineering  applications, knowing the internal state of a system is mandatory to solve fundamental problems, such as state feedback stabilization and fault detection. However, most of the times the state of the system is not accessible, and hence one needs to design a suitable observer that  produces, at least asymptotically, a good estimate of  the original state vector.
The theory of asymptotic observers originated with the  works of Luenberger \cite{lue,lue2},
focusing on linear state-space models.
In the standard set-up, the model   description as well as the input and output signals affecting the system are assumed to be available.
In a lot of practical situations, however, 
the system dynamics is affected by disturbances,  measurement errors  or other unknown signals 
that cannot be used to identify the state evolution.
 Therefore, in the last decades considerable attention has been devoted to study the problem of state estimation in the presence of unknown inputs. The goal is to design an observer whose estimation error asymptotically converges to zero, regardless  of the initial conditions, 
 and  of the dynamics of the unknown inputs  acting on the system. This can be considered a qualitative definition of {\em unknown-input observer (UIO)}, which is the core of   this paper. 

In the literature we can find numerous solutions to the problem, exploiting different approaches: some use a priori information about the unknown input, for instance by modeling it as the response of a suitably chosen dynamical system \cite{Hostetter}, others instead assume to have no prior knowledge on the unknown disturbance and solve the problem trying to exploit decoupling properties of the system using algebraic methods \cite{Hou-Muller1, Kudva, doppioWang}, geometric methods \cite{Bhatta}, generalized inverse approaches \cite{Miller} or techniques based on the singular value decomposition \cite{Fairman}, just to mention a few. Necessary and sufficient conditions for the existence of a UIO have been derived (see, e.g., \cite{Darouach2,Darouach,UIO-MEV}) and practical design procedures have been provided, e.g., in \cite{Yang-Wilde}. 
All the works mentioned so far 
rely on the common assumption that   the system model is known, and therefore their analysis is carried out using model-based approaches. 

{\color{black} More recently, the availability of large quantities of data has led to an increasingly widespread diffusion of data-driven techniques to solve control engineering problems \cite{DePersisTesi,MarkRap}, including the state estimation problem \cite{Mishra22,Shi22,Ferrari-Trecate,Wolff22}.
Two types of techniques have been adopted: a two-step approach, that relies on a preliminary system identification step, and a single step approach, that exploits directly the collected data, avoiding the identification phase. 
However, in some cases (see, e.g., \cite{Ferrari-Trecate})  it is not possible to uniquely identify the system leveraging only the available data and thus  a one-step procedure is the only viable option. }
\\
In this paper we  {\color{black} consider a problem set-up similar to the one adopted in \cite{Ferrari-Trecate} and hence  focus on a single step data-driven approach. } 
More in detail, the goal of this paper is to determine necessary and sufficient conditions for the existence  of an unknown-input observer for a discrete-time linear time-invariant (LTI) system, designed based only on some available data (obtained on a finite time window), without exploiting the knowledge of the system matrices. 
The problem of designing a data-driven UIO for this type of system has already been addressed in the literature in \cite{Ferrari-Trecate} and {\em sufficient conditions} for its existence have been derived. 
Indeed, under suitable assumptions, the collected data {\color{black} have  been used in  \cite{Ferrari-Trecate}} to derive the state space description of {\em one of the candidate UIOs}. 
If  such system is asymptotically stable, then it is a  UIO  that asymptotically tracks the  
state of the original system, despite   the presence of disturbances. {\color{black} However, if the obtained system is not asymptotically stable, it is  not obvious if a UIO can be designed based on such data.}
Compared with  \cite{Ferrari-Trecate}, our contribution is threefold: (1) 
 we  provide necessary and sufficient conditions for the problem solvability {\color{black} that can be verified a priori on data};
 (2) we provide a complete parametrization of all candidate UIOs;
 (3) we   prove that, under certain hypotheses {\color{black} on the collected data}, the solvability conditions derived by means of the data-driven approach are identical to those obtained through the model-based one in \cite{Darouach2,Darouach,Yang-Wilde}.
 
The paper is organized as follows. Section \ref{problem} provides  the formal  problem statement. Section \ref{sec:MBapproach} examines the model-based approach, providing   necessary and sufficient conditions for the existence of a UIO.
Section \ref{DD_UIO} provides the solution to the problem in the data-driven framework, giving a complete parametrization of all possible UIOs. 
{\color{black} In Section \ref{concl_rem} some useful remarks about how to simplify the problem solution, as well as a numerical example, 
are given.} Finally, Section \ref{concl} concludes the paper.

\smallskip
{\bf Notation.} 
Given a matrix $M\in {\mathbb R}^{p \times m}$, we denote by $M^\dag\in {\mathbb R}^{m \times p}$ its {\em Moore-Penrose inverse} \cite{BenIsraelGreville}. Note that if $M$ is of full column rank, then $M^\dag = (M^\top M)^{-1} M^\top$. A symmetric result holds if $M$ is of full row rank. 
The null and column space of $M$ are denoted by $\ker{(M)}$ and $\im (M)$, respectively. 
Given a vector signal $v(t)\in\mathbb{R}^n$ with $t\in \mathbb Z _+$, we use the notation $\{v(t)\}_{t=0}^{N}$, $N\in\mathbb Z _+$, to indicate the sequence of vectors $v(0),\dots,v(N)$.  
\medskip
 
\section{Problem formulation}\label{problem}

Consider the discrete-time LTI system, $\Sigma$, described by:
\begin{eqnarray}\label{system_1}
x(t+1) &=& Ax(t)+Bu(t)+Ed(t) \\ 
y(t) &=& Cx(t), \label{system_2}
\end{eqnarray}
where $t\in\mathbb{Z}_+$, $x(t)\in \mathbb{R}^n$ is the state, $u(t)\in\mathbb{R}^m$ is the (known) control input, $y(t)\in\mathbb{R}^p$ is the output and $d(t)\in\mathbb{R}^r$ is the unknown input of the system, e.g., a disturbance. Without loss of generality, 
we assume that the matrix $E\in\mathbb{R}^{n\times r}$ is of full column rank, i.e., $\rank{E}=r$. Indeed, if 
$\rank{E}=\bar r<r$, we can always rewrite it as $E =\bar{E} T $, where $\bar{E}\in\mathbb{R}^{n\times \bar r}$ is a full column rank matrix and $T\in\mathbb{R}^{\bar r \times r}$ is a full row rank matrix, and define a new unknown input $\bar{d} (t) \triangleq Td(t)$. 

A UIO for system \eqref{system_1}-\eqref{system_2}  is a state space model,  
receiving as its inputs   the input and output of the original system and producing as its output an    estimate $\hat x$ of the state $x$ of  \eqref{system_1}-\eqref{system_2}, such that 
$e(t) \triangleq x(t)-\hat x(t)$ (the estimation error) asymptotically converges to zero, regardless of the initial conditions and of the dynamics of the unknown input  acting on the system.  
More specifically, in the sequel we will refer to the following definition of UIO.
\smallskip

\begin{definition}\label{UIO}
An LTI system $\hat{\Sigma}$ of the form
\begin{eqnarray}\label{UIO_eq1}
z(t+1) &=& A_{UIO}z(t)+B_{UIO}^u u(t)+B_{UIO}^y y(t) \\
\hat{x}(t) &=& z(t) + D_{UIO}y(t), \label{UIO_eq2}
\end{eqnarray}
where   $z(t)$ and $\hat x(t)$, both belonging to  $\mathbb{R}^n$,  are the state and the output of $\hat \Sigma$, respectively,  is an  {\em  unknown-input observer (UIO)} for the system in \eqref{system_1}-\eqref{system_2} if $e(t) \triangleq x(t) - \hat{x} (t)$ tends to 0 as $t\to +\infty$, for every choice of $x(0)$, $z(0)$ and the input signal  $u(t), t\in {\mathbb Z}_+$, and   independently of the unknown input $d(t), t\in {\mathbb Z}_+$. 
\end{definition}
\medskip

\section{Necessary and sufficient conditions for the existence of a UIO: model-based approach}
\label{sec:MBapproach}

In this section we briefly recall the necessary and sufficient conditions for the existence
of a UIO $\hat \Sigma$ for system $\Sigma$ first derived in \cite{Darouach2, Darouach}.
By making use of the system and UIO descriptions, we easily deduce that
 the state estimation error obeys the following dynamics:  \\
\\ $e(t+1) = x(t+1)-\hat{x} (t+1)$
\vspace*{-0.12cm}
\begin{eqnarray*} \label{error_dyn}
\!\!\!&=&\!\!\! x(t+1) - z(t+1) - D_{UIO}y(t+1) \\
\!\!\!&=&\!\!\! x(t+1) - A_{UIO} z(t)- B_{UIO}^u u(t) - B_{UIO}^y y(t) \\
\!\!\!&-&\!\!\!   D_{UIO}C x(t+1) \\ 
\!\!\!&=&\!\!\! (I-D_{UIO}C) x(t+1)  - A_{UIO} \hat x(t) - B_{UIO}^u u(t) \\
\!\!\!&+&\!\!\!   [A_{UIO} D_{UIO} - B_{UIO}^y] Cx(t)  \\ 
\!\!\!&=&\!\!\! A_{UIO}e(t) \\
\!\!\!&+&\!\!\!   (I-D_{UIO}C)Ed(t) + [(I-D_{UIO}C)B-B_{UIO}^u]u(t)  
\\
\!\!\!&+&\!\!\!   [(I-D_{UIO}C)A-A_{UIO}(I-D_{UIO}C)-B_{UIO}^yC]x(t).
\end{eqnarray*}
Therefore,   $e(t)$ is independent of the disturbance $d(t)$ and tends to 0 as $t\to +\infty$, for every choice of $u(t), t\in {\mathbb Z}_+$, $x(0)$ and $z(0)$, {\color{black} if and only if there exist $A_{UIO},
 B^u_{UIO}, B^y_{UIO}$, and $D_{UIO}$ such that the following conditions are satisfied}: 
\begin{eqnarray}
\!\!\!\!\!\!\!\!\!\!\!\!&&A_{UIO} \text{ \ is Schur stable}, \label{SS_cond1} \\
\!\!\!\!\!\!\!\!\!\!\!\!&&D_{UIO}CE = E, \label{SS_cond2} \\
\!\!\!\!\!\!\!\!\!\!\!\!&&B_{UIO}^u=(I-D_{UIO}C)B, \label{SS_cond3} \\
\!\!\!\!\!\!\!\!\!\!\!\!&&A_{UIO}(I-D_{UIO}C) + B_{UIO}^y  C = (I-D_{UIO}C)A.  \label{SS_cond4}
\end{eqnarray}
{\color{black}
When so,   the state estimation error follows the autonomous asymptotically stable dynamics 
$$e(t+1) = A_{UIO}e(t).$$

In the next theorem, we summarize the necessary and sufficient conditions for the existence of a UIO available in the literature. The proof is omitted since it can be   obtained by putting together  Theorems 1 and 2 in \cite{Darouach}, and Theorem 4 in \cite{Darouach2}.

\begin{theorem}\label{cns_MB}
The following facts are equivalent.
\begin{itemize}
\item[(i)] There exists a UIO $\hat \Sigma$ of the form \eqref{UIO_eq1}-\eqref{UIO_eq2} for system $\Sigma$.
\item[(ii)] There exist matrices $A_{UIO}\in {\mathbb R}^{n\times n}, B_{UIO}^u\in {\mathbb R}^{n \times m}, 
B_{UIO}^y\in {\mathbb R}^{n \times p}, $ and $D_{UIO} \in {\mathbb R}^{n \times p}$ that satisfy conditions
\eqref{SS_cond1}$\div$\eqref{SS_cond4}.
\item[(iii)] The following two conditions hold: 
\begin{itemize}
\item[(a)] ${\rm rank} (CE)={\rm rank}(E)=r$, and
\item[(b)] $
{\rm rank} \begin{bmatrix} zI_n - A & -E\cr
C & 0\end{bmatrix} = n+r, \ \forall z\in {\mathbb C}, |z|\ge 1$.
\end{itemize}
\item[(iv)] The triple $(A,E,C)$  is strong* detectable 
(see Definition 2 in \cite{Darouach2}), meaning that  $\lim_{t\to+\infty}{y(t)}=0$ implies $\lim_{t\to+\infty}{x(t)}=0$ for all $d(t)$ and $x(0)$, when $u=0$.\end{itemize}
\end{theorem}

\begin{remark} \label{remark3} It is worth noticing that condition  (iii), point  (a), alone, is equivalent to the existence of 
matrices $A_{UIO}, B_{UIO}^u, 
B_{UIO}^y, $ and $D_{UIO}$ that satisfy conditions
\eqref{SS_cond2}$\div$\eqref{SS_cond4}.
By adding condition   (iii), point   (b), we can guarantee that 
among the solutions of \eqref{SS_cond2}$\div$\eqref{SS_cond4} there is at least one with $A_{UIO}$ Schur stable.
\end{remark}}
%
%
\medskip

\section{The data-driven approach} \label{DD_UIO}
In order to tackle the problem in the data-driven framework, we assume (as in \cite{Ferrari-Trecate})
that we have performed an \emph{offline} experiment where we have collected some input/output/state trajectories in the time interval $[0,T-1]$ with $T\in\mathbb{Z}_+$, and we define the following vector sequences  
 $u_d =  \{u_d(t)\}_{t=0}^{T-2}$,  $y_d = \{y_d(t)\}_{t=0}^{T-1}$ and  $x_d =  \{x_d(t)\}_{t=0}^{T-1}$,
where we used the subscript $d$ to highlight the fact that we are referring to precollected (i.e., \emph{historical}) data. 
{\color{black} The motivation behind the assumption to have access to the state during the preliminary offline measurements is twofold
(see, also, \cite{Ferrari-Trecate} for a detailed discussion).
On the one hand, the access to the state in standard working conditions  may be not advisable, due to security reasons or to the high costs of dedicated sensors. However, this may become possible in a lab, in a dedicated test.
On the other hand, 
the only way to design a UIO from data is to have some information about the state itself.
  Indeed, it would not be possible to uniquely identify the state of the system and hence to construct a  UIO only from input/output data, without any knowledge of the dimension and the basis  of the state-space. The same input/output data are compatible with an infinite number of state-space models (even under reachability and observability assumptions) and hence do not provide sufficient information on the system to allow one to estimate its state.
Therefore even if it seems a restrictive assumption, the knowledge of some historical state measurements  is in fact necessary 
for the design of a data-driven UIO. 

Finally, even if  the unknown input 
is not accessible, and therefore we do not assume that disturbance data are available,  for the subsequent analysis it is useful to introduce a symbol for the sequence of historical unknown input data, i.e.,
  $d_d = \{d_d(t)\}_{t=0}^{T-2}$. \\
We rearrange 
the above data into the following matrices: 
\begin{eqnarray*}
U_p &\triangleq &\begin{bmatrix}
u_d(0) & \dots & u_d(T-2)
\end{bmatrix} \in {\mathbb R}^{m \times (T-1)}, \\
X_p &\triangleq& \begin{bmatrix}
x_d(0) & \dots & x_d(T-2)
\end{bmatrix} \in {\mathbb R}^{n \times (T-1)}, \\
X_f &\triangleq &\begin{bmatrix}
x_d(1) & \dots & x_d(T-1)
\end{bmatrix} \in {\mathbb R}^{n \times (T-1)}, \\
Y_p &\triangleq &\begin{bmatrix}
y_d(0) & \dots & y_d(T-2)
\end{bmatrix}\in {\mathbb R}^{p \times (T-1)}, \\
Y_f &\triangleq &\begin{bmatrix}
y_d(1) & \dots & y_d(T-1)
\end{bmatrix} \in {\mathbb R}^{p \times (T-1)}, 
\\
D_p &\triangleq& \begin{bmatrix}
d_d(0) & \dots & d_d(T-2)
\end{bmatrix} \in {\mathbb R}^{r \times (T-1)},
\end{eqnarray*}
where the subscripts $p$ and $f$ stand for past and future, respectively. }
Before providing the data-driven UIO formulation, we give the following definition, which is a slight modification of the one given in \cite{Ferrari-Trecate}.

\begin{definition}\label{compatible_traj}
{\color{black} An (input/output/state) trajectory} $(\{u(t)\}_{t\in {\mathbb Z}_+},$ $\{y(t)\}_{t\in {\mathbb Z}_+},\{x(t)\}_{t\in {\mathbb Z}_+})$  is said to be {\em compatible with the historical data} $(u_d, y_d, x_d)$ if 
\be\label{compatibility}
\begin{bmatrix}
u(t) \\
y(t) \\
x(t) \\
x(t+1) 
\end{bmatrix} \in 
\im \left(
\begin{bmatrix}
U_p \\
Y_p \\
X_p \\
X_f 
\end{bmatrix}
\right), \ \forall t\in {\mathbb Z}_+.
\ee
\end{definition}

The set of all   trajectories compatible with the historical data $(u_d, y_d, x_d)$ is denoted by
\begin{eqnarray}\label{comp_set}
\mathbb{T}_c(u_d,y_d,x_d) \!\!\!\! &\triangleq& \!\!\!\! \{ (\{u(t)\}_{t\in {\mathbb Z}_+},\{y(t)\}_{t\in {\mathbb Z}_+},\{x(t)\}_{t\in {\mathbb Z}_+}) : \nonumber \\ 
\!\!\!\!&&\!\!\!\! \eqref{compatibility}\text{ holds}\}.
\end{eqnarray}

\begin{remark} It is worth noticing that the definition of compatibility that we adopt is slightly different from the one introduced in \cite{Ferrari-Trecate} (see Definition 2) in that we  have replaced a condition on the vector $\begin{bmatrix}
u(t)^\top \!&\!
y(t)^\top  \!&\!
x(t)^\top  \!&\!
u(t+1)^\top  \!&\!
y(t+1)^\top  \!&\!
x(t+1)^\top
\end{bmatrix}^\top$ with one on 
$\begin{bmatrix}
u(t)^\top &
y(t)^\top &
x(t)^\top &
x(t+1)^\top
\end{bmatrix}^\top$.
As we will see, 
this definition is equally powerful  when 
trying to identify (based on the historical data)
  the trajectories that are compatible with the system, but is more compact.
  {\color{black} Moreover, instead of imposing condition \eqref{comp_set} for $0 \le t \le T-2$, we believe 
  that checking it on ${\mathbb Z}_+$ better formalises the idea that 
  a finite set of  historical data can be used to characterise system trajectories defined on the whole (nonnegative) time axis.
  }
\end{remark}

We now introduce   the set of all the (input/output/state) trajectories that can be generated by the system in \eqref{system_1}-\eqref{system_2} (corresponding to some disturbance sequence): 
\begin{eqnarray}\label{system_traj}
\mathbb{T}_\Sigma \!\!\!\! &\triangleq& \!\!\!\! \{(\{u(t)\}_{t\in {\mathbb Z}_+},\{y(t)\}_{t\in {\mathbb Z}_+},\{x(t)\}_{t\in {\mathbb Z}_+}) : \exists \{d(t)\}_{t\in {\mathbb Z}_+}\nonumber \\ 
\!\!\!\!\!\!\!\!&&\!\!\!\!\text{s.t.}  (\{u(t)\}_{t\in {\mathbb Z}_+},\{y(t)\}_{t\in {\mathbb Z}_+},\{x(t)\}_{t\in {\mathbb Z}_+}, \{d(t)\}_{t\in {\mathbb Z}_+}) \nonumber \\
\!\!\!\!&&\!\!\!\!\text{satisfies } \eqref{system_1}-\eqref{system_2}, \ \forall t \in {\mathbb Z}_+\}.
\end{eqnarray}
In order to be able to design a data-driven UIO, we want the historical data to be representative of the system. Therefore, our aim is to perform an experiment so that the two sets defined above actually coincide. 

All the subsequent analysis is carried out under the following: 

\smallskip
\noindent {\bf Assumption:}
{\color{black} The size $r$ of the unknown input is known and }the matrix $\begin{bmatrix}
U_p^\top &
D_p^\top &
X_p^\top 
\end{bmatrix}^\top$ is of full row rank, {\color{black} i.e., $m+r+n$}.  
\smallskip 

{\color{black} 
\begin{remark}
Clearly, as  the unknown input is not measurable,   
the previous assumption cannot be checked in practice.
 However, it is still reasonable to assume that 
 an offline experiment can be designed in such a way that 
 it holds. Indeed, if the system is reachable \cite{Kailath},
and  the historical data $(\{u_d(t)\}_{t=0}^{T-2} \ \{d_d(t)\}_{t=0}^{T-2})$ are persistently exciting of order $n+1$, then by Corollary 2 in \cite{WillemsPE} the  Assumption holds. The control input can be chosen to this purpose, and for random disturbances 
this property generically holds. {\color{black} For what concerns the dimension of the unknown input, since it is related to the rank of the matrices of the input, state and output data, that are available, it can be deduced by performing repeated experiments and computing the (max) rank of matrices of the collected input, output and state data. }
\end{remark} }

\begin{lemma}\label{Tc_equiv_T} Under the Assumption on the collected data,
 the trajectories generated by the system $\Sigma$ in \eqref{system_1}-\eqref{system_2} are all and only those compatible with the given historical data, i.e., 
$$ \mathbb{T}_\Sigma = \mathbb{T}_c(u_d,y_d,x_d). $$
\end{lemma}

\begin{proof} The proof bears similarities to the proof of Lemma 1 in \cite{Ferrari-Trecate}, but as previously mentioned our definition of $\mathbb{T}_c(u_d,y_d,x_d)$ is different. So, the proof is here provided for the sake of completeness.
{\color{black} We preliminarily observe that a triple $(\{u(t)\}_{t\in {\mathbb Z}_+},\{y(t)\}_{t\in {\mathbb Z}_+},\{x(t)\}_{t\in {\mathbb Z}_+})$ is a trajectory of $\Sigma$ if and only if it satisfies the following equation 
\be
\begin{bmatrix}
u(t) \\
y(t) \\
x(t) \\
x(t+1) 
\end{bmatrix} = 
\begin{bmatrix}
I & 0 & 0 \\
0 & 0 & C \\
0 & 0 & I \\
B & E & A
\end{bmatrix}
\begin{bmatrix}
u(t) \\
d(t) \\
x(t) 
\end{bmatrix}, \ \forall t \in \mathbb{Z}_+,
\label{prima}
\ee
for some $\{d(t)\}_{t\in {\mathbb Z}_+}$.
As the historical data have been generated by the system $\Sigma$, it clearly holds that 
\be
\begin{bmatrix}
U_p \\
Y_p \\
X_p \\
X_f
\end{bmatrix} = 
\begin{bmatrix}
I & 0 & 0 \\
0 & 0 & C \\
0 & 0 & I \\
B & E & A
\end{bmatrix}
\begin{bmatrix}
U_p \\
D_p \\
X_p 
\end{bmatrix}. 
\label{BEAdati}
\ee
We first show that $ \mathbb{T}_\Sigma \supseteq \mathbb{T}_c(u_d,y_d,x_d). $ 
If  $(\{u(t)\}_{t\in {\mathbb Z}_+},$ $\{y(t)\}_{t\in {\mathbb Z}_+},\{x(t)\}_{t\in {\mathbb Z}_+})  \in \mathbb{T}_c(u_d,y_d,x_d)$, then for every $t\in \mathbb{Z}_+$ there exists $g_t\in \mathbb{R}^{T-1}$ such that
\be
\begin{bmatrix}
u(t) \\
y(t) \\
x(t) \\
x(t+1) 
\end{bmatrix} = 
\begin{bmatrix}
U_p \\
Y_p \\
X_p \\
X_f
\end{bmatrix} g_t.
\label{gt}
\ee
Therefore, by making use of \eqref{BEAdati}, we get that \eqref{prima} holds for $d(t) = D_p g_t$.
Thus, $ \mathbb{T}_\Sigma \supseteq \mathbb{T}_c(u_d,y_d,x_d). $ }

We now prove that also the other inclusion holds, namely $ \mathbb{T}_\Sigma \subseteq \mathbb{T}_c(u_d,y_d,x_d). $ 
Since $\begin{bmatrix}
U_p^\top &
D_p^\top &
X_p^\top 
\end{bmatrix}^\top$ is of full row rank by Assumption, {\color{black} it defines a surjective map and hence} for every trajectory $(\{u(t)\}_{t\in {\mathbb Z}_+},\{y(t)\}_{t\in {\mathbb Z}_+},$ $\{x(t)\}_{t\in {\mathbb Z}_+})\in \mathbb{T}_\Sigma$ there  exists $\{g_t\}_{t\in {\mathbb Z}_+}$, taking values in $\mathbb{R}^{T-1}$, such that 
$$\begin{bmatrix}
u(t) \\
d(t) \\
x(t)  
\end{bmatrix} = \begin{bmatrix}
U_p \\
D_p \\
X_p  
\end{bmatrix}g_t, \qquad \forall t\in {\mathbb Z}_+.$$
Therefore, for every trajectory of $\Sigma$ we have  {\color{black}(by \eqref{BEAdati})}
\begin{eqnarray*}
\begin{bmatrix}
u(t) \\
y(t) \\
x(t) \\
x(t+1) 
\end{bmatrix} &=&  
\begin{bmatrix}
I & 0 & 0 \\
0 & 0 & C \\
0 & 0 & I \\
B & E & A
\end{bmatrix}
\begin{bmatrix}
u(t) \\
d(t) \\
x(t) 
\end{bmatrix} \\
&=& 
\begin{bmatrix}
I & 0 & 0 \\
0 & 0 & C \\
0 & 0 & I \\
B & E & A
\end{bmatrix} 
\begin{bmatrix}
U_p \\
D_p \\
X_p  
\end{bmatrix}g_t = \begin{bmatrix}
U_p \\
Y_p \\
X_p \\
X_f
\end{bmatrix} g_t, 
\end{eqnarray*}
which implies that 
{\color{black}$(\{u(t)\}_{t\in {\mathbb Z}_+},\{y(t)\}_{t\in {\mathbb Z}_+},\{x(t)\}_{t\in {\mathbb Z}_+}) \in  \mathbb{T}_c(u_d,y_d,x_d)$ and hence $\mathbb{T}_\Sigma \subseteq \mathbb{T}_c(u_d,y_d,x_d).
$}
\end{proof}

Let us define now the set of all the {\color{black} (input/output)} trajectories generated by the system  \eqref{UIO_eq1}-\eqref{UIO_eq2} as 
\begin{eqnarray}\label{UIO_traj}
\mathbb{T}_{\hat{\Sigma}} \!\!\!\!\! &\triangleq& \!\!\!\!\! \{(\{u(t)\}_{t\in {\mathbb Z}_+},\{y(t)\}_{t\in {\mathbb Z}_+},\{\hat x(t)\}_{t\in {\mathbb Z}_+}): \exists \{z(t)\}_{t\in {\mathbb Z}_+}\nonumber \\ 
\!\!\!\!&&\!\!\!\!\text{s.t.}  (\{u(t)\}_{t\in {\mathbb Z}_+},\{y(t)\}_{t\in {\mathbb Z}_+},\{\hat x(t)\}_{t\in {\mathbb Z}_+}, \{z(t)\}_{t\in {\mathbb Z}_+})  \nonumber \\
\!\!\!\!\!&&\!\!\!\!\! \text{ satisfies } \eqref{UIO_eq1}-\eqref{UIO_eq2} \  \forall t \in {\mathbb Z}_+\}.
\end{eqnarray}
{\color{black} In the following proposition we provide necessary and sufficient conditions based on the historical data to guarantee that there exists a system $\hat \Sigma$, described as in \eqref{UIO_eq1}-\eqref{UIO_eq2}, such that all the
trajectories of   $\mathbb{T}_{\Sigma}$ are also trajectories of $\mathbb{T}_{\hat \Sigma}$. Moreover, we also relate such conditions to the rank constraint given in {\em (iii)}, point {\em (a)}, of Theorem \ref{cns_MB}. 
Despite the equivalence of {\em (i)} and  {\em (iii)} has been proved  in \cite{Ferrari-Trecate} (see, Lemma 2), 
it is here derived passing through condition   {\em (ii)}, that will  lead to 
 a parametrization of {\em all}  quadruples $(A_{UIO}, B_{UIO}^u, B_{UIO}^y, D_{UIO})$ describing a possible UIO for $\Sigma$ (see Corollary \ref{param}).
This is one of the major contributions of this paper
compared with \cite{Ferrari-Trecate}, where a single quadruple  $(A_{UIO}, B_{UIO}^u, B_{UIO}^y, D_{UIO})$   is derived from the historical data (see the comment about the uniqueness, after the proof of Theorem 1 in \cite{Ferrari-Trecate}).  
 Such specific quadruple $(A_{UIO}, B_{UIO}^u, B_{UIO}^y, D_{UIO})$
 represents a UIO if and only if 
  the matrix $A_{UIO}$ is Schur stable. However,  if this is not the case, it is not clear if the matrices of a UIO can be found by other means.
  In this paper, instead, we will prove  
  that if a UIO exists, then its matrices  
  can be found in our parametrization.  }
\medskip

\begin{proposition}\label{kernelincl}
Under the Assumption on the data,
{\color{black} the following facts are equivalent.
\begin{itemize}
\item[(i)] There exists a system $\hat \Sigma$ of the form \eqref{UIO_eq1}-\eqref{UIO_eq2} such that $\mathbb{T}_\Sigma \subseteq\mathbb{T}_{\hat{\Sigma}}$.
\item[(ii)] $\exists 
 \left[\begin{array}{c|c|c|c}
 T_1 & T_2 & T_3 & T_4
 \end{array}
 \right]\in {\mathbb R}^{n \times (m+2p+n)}$ s.t. 
 \be\label{Xf}
 X_f =  \left[\begin{array}{c|c|c|c}
 T_1 & T_2 & T_3 & T_4
 \end{array}
 \right]
\begin{bmatrix}
U_p \\
Y_p \\
Y_f \\
X_p 
\end{bmatrix}.
\ee 
\item[(iii)]
\be\label{ker_inclusion}
\ker{(X_f)}\supseteq \ker{\left(\begin{bmatrix}
U_p \\
Y_p \\
Y_f \\
X_p 
\end{bmatrix}\right)}.
\ee
\item[(iv)]   $\Sigma$ satisfies condition $\rank (CE)=\rank (E)=r$.
\end{itemize}
 }\end{proposition}

\begin{proof} {\color{black} The equivalence of {\em (ii)} and {\em (iii)} follows from standard Linear Algebra.}\\
{\em (i)} $\Rightarrow$ {\em (ii)}\ 
 Suppose that there exists a system of the form \eqref{UIO_eq1}-\eqref{UIO_eq2} such that every trajectory $(\{u(t)\}_{t\in {\mathbb Z}_+},\{y(t)\}_{t\in {\mathbb Z}_+},\{x(t)\}_{t\in {\mathbb Z}_+}) \in  \mathbb{T}_\Sigma$ 
 satisfies also equations \eqref{UIO_eq1}-\eqref{UIO_eq2}, which means  that there exists $\{z(t)\}_{t\in {\mathbb Z}_+}$ s.t. $\forall t\in {\mathbb Z}_+$
\begin{eqnarray*}
z(t+1) &=& A_{UIO}z(t)+B_{UIO}^u u(t)+B_{UIO}^y y(t) \\
x(t) &=& z(t) + D_{UIO}y(t).
\end{eqnarray*}
This 
 holds, in particular, for the historical data $(u_d, y_d, x_d)$, implying that $\exists\  Z_p \triangleq \begin{bmatrix}
z(0) & \dots & z(T-2)
\end{bmatrix}$ and $Z_f \triangleq \begin{bmatrix}
z(1) & \dots & z(T-1)
\end{bmatrix}$, s.t. 
\begin{eqnarray*}
Z_f &=& A_{UIO}Z_p+B_{UIO}^u U_p+B_{UIO}^y Y_p \\
X_p &=& Z_p + D_{UIO}Y_p \\
X_f &=& Z_f +D_{UIO} Y_f. 
\end{eqnarray*}
This, in turn, implies that 
{\small\begin{align}
&X_f = A_{UIO}Z_p+B_{UIO}^u U_p+B_{UIO}^y Y_p +D_{UIO}Y_f \nonumber \\
\!\!&= A_{UIO}(X_p-D_{UIO}Y_p)+B_{UIO}^u U_p+B_{UIO}^y Y_p +D_{UIO}Y_f \nonumber \\
\!\!&=  \begin{bmatrix}
B_{UIO}^u \!&\! B_{UIO}^y-A_{UIO}D_{UIO} \!&\! D_{UIO} \!&\! A_{UIO}
 \end{bmatrix}\!\!
\begin{bmatrix}
U_p \\
Y_p \\
Y_f \\
X_p 
\end{bmatrix}\!\! \label{ciccio}
\end{align}}
and hence {\em (ii)} holds {\color{black}  for 
\begin{eqnarray}
T_1= B_{UIO}^u, && T_2 = B_{UIO}^y-A_{UIO}D_{UIO}, \label{daUIOaT1}\\
T_3= D_{UIO}, && T_4= A_{UIO}.
\label{daUIOaT2}
\end{eqnarray}

\noindent {\em (ii)} $\Rightarrow$ {\em (iv)} \
Suppose that \eqref{Xf} holds for suitable matrices $T_1, T_2, T_3,$ and $T_4.$
Since the data matrices $X_p, X_f, U_p, Y_p$ and $Y_f$ are generated by the system $\Sigma$, we can write
\be
X_f = [B \ | \ E \ | \ A]\begin{bmatrix}
U_p \\
D_p \\
X_p 
\end{bmatrix}
\label{new1}
\ee
as well as
\be
\begin{bmatrix}
U_p \\
Y_p \\
Y_f \\
X_p 
\end{bmatrix}
= \begin{bmatrix}
I_m & 0 & 0\cr 0 & 0 & C\cr
CB & CE & CA \cr
0 & 0 &I_n
\end{bmatrix}
 \begin{bmatrix}
U_p \\
D_p \\
X_p 
\end{bmatrix}.
\label{new2}
\ee
By replacing \eqref{new1} and \eqref{new2} in \eqref{Xf}   and by exploiting the Assumption, we deduce the identity
$$ [B \ | \ E \ | \ A] = 
\left[\begin{array}{c|c|c|c}
 T_1 & T_2 & T_3 & T_4
 \end{array}
 \right] \begin{bmatrix}
I_m & 0 & 0\cr 0 & 0 & C\cr
CB & CE & CA \cr
0 & 0 &I_n
\end{bmatrix},$$
which implies, in particular, that 
$E = T_3 CE$ and hence condition {\em (iv)} holds.
\smallskip

\noindent {\em (iv)} $\Rightarrow$ {\em (i)} \ If $\rank (CE)=\rank (E)=r$,
then there exist matrices $A_{UIO}, B_{UIO}^u, B_{UIO}^y$ and $D_{UIO}$
satisfying conditions \eqref{SS_cond2}$\div$\eqref{SS_cond4} (but not necessarily \eqref{SS_cond1}) (see Remark \ref{remark3}).
We want to prove that under these assumptions on its describing matrices, the system  $\hat \Sigma$ of equations 
\eqref{UIO_eq1}-\eqref{UIO_eq2} satisfies $T_{\Sigma} \subseteq T_{\hat \Sigma}$.
Clearly,  conditions \eqref{SS_cond2}$\div$\eqref{SS_cond4} ensure that
$e(t)= x(t) - \hat x(t)$ updates according to equation
$e(t+1)=A_{UIO} e(t)$. So,   proving that $T_{\Sigma} \subseteq T_{\hat \Sigma}$ amounts to proving that it is possible to choose $z(0)$ so that 
$e(0)= x(0)- \hat x(0)=0$.
In fact, by assuming  $z(0) = x(0) - D_{UIO} y(0)$  (see Remark 2 in \cite{Ferrari-Trecate})
we obtain $e(0)=0$.
}
\end{proof} 

To summarize, if the data we have collected satisfy the following conditions:\\
(a) The matrix $\begin{bmatrix}
U_p^\top &
D_p^\top &
X_p^\top 
\end{bmatrix}^\top$ is of full row rank; \\
(b) $\ker{(X_f)}\supseteq \ker{\left(\begin{bmatrix}
U_p \\
Y_p \\
Y_f \\
X_p 
\end{bmatrix}\right)}$, \\ then
we can construct a potential  UIO  described as in \eqref{UIO_eq1}-\eqref{UIO_eq2} for the original system.
In fact, such a system was called an {\em acceptor} in \cite{BehaviorObservers}, to explain the general concept  that
the system, given the available information, should not introduce additional constraints on the variable to be estimated other than those imposed by the original system itself. This amounts to saying that 
 if $(u,y,x)$ is an  input/output/state trajectory generated by the system in \eqref{system_1}-\eqref{system_2}, then corresponding to the input pair $(u,y)$ (the available information) the acceptor $\hat \Sigma$ should have $\hat x = x$ as one of its possible outputs.
 
An acceptor is not necessarily a UIO. For this to happen, we need to ensure also that if $(u,y, \hat x_1)$ and $(u,y,\hat x_2)$ are two trajectories in $\mathbb{T}_{\hat{\Sigma}}$, then $\lim_{t\to +\infty} \hat x_2(t)-\hat x_1(t)=0$.
This is the final step that will be addressed in Theorem \ref{UIO_final} below.

{\color{black}
\begin{theorem}\label{UIO_final}
Under the Assumption on the data, the following  facts are equivalent.
\begin{itemize}
\item[(i)] There exists a UIO $\hat \Sigma$ of the form \eqref{UIO_eq1}-\eqref{UIO_eq2} such that  $\mathbb{T}_\Sigma \subseteq\mathbb{T}_{\hat{\Sigma}}$.
 \item[(ii)] There exist matrices $T_1,T_2,T_3, T_4$ of suitable sizes such that \eqref{Xf} holds and $T_4$ is Schur stable.
\item[(iii)] Condition \eqref{ker_inclusion}  and condition
\be
{\rm rank} \begin{bmatrix}
z X_p - X_f\cr U_p \cr Y_p\end{bmatrix} = n+m +r, \ \forall\ z\in {\mathbb C}, |z|\ge 1,
\label{dd_rank}\ee
hold.
\item[(iv)] The triple $(A,E,C)$ is strong* detectable.
\end{itemize}
\end{theorem}

\begin{proof}
\emph{(i)} $\Rightarrow$ \emph{(ii)}. If there exists a UIO $\hat \Sigma$ of the form \eqref{UIO_eq1}-\eqref{UIO_eq2}  such that  $\mathbb{T}_\Sigma \subseteq\mathbb{T}_{\hat{\Sigma}}$, then we can refer to the proof of Proposition \ref{kernelincl}   to claim that 
\eqref{Xf}  holds (see \eqref{ciccio})  with $T_1, T_2, T_3$ and $T_4$ as in 
\eqref{daUIOaT1} and \eqref{daUIOaT2}. Since $T_4=A_{UIO}$, clearly $T_4$ is Schur stable.

\noindent
\emph{(ii)} $\Rightarrow$ \emph{(iii)}.  
From Proposition \ref{kernelincl}  we know that 
the existence of $T_1,T_2,T_3, T_4$ such that \eqref{Xf} holds implies that 
\eqref{ker_inclusion} holds and that $\rank (CE)=\rank (E)=r$. To prove the second part of {\em (iii)},
 we preliminarily show that 
$$
\rank{\begin{bmatrix}
zX_p-X_f \\ U_p \\ Y_p
\end{bmatrix}} = 
\rank{\begin{bmatrix}
zX_p-X_f \\ U_p \\ Y_p\\
Y_f 
\end{bmatrix}}, \ \forall z \in \mathbb C.$$ 
Indeed, for every $z\in\mathbb C$ 
\begin{eqnarray*}
\!\!\!\!\!\!\!\!\!\!\!\!&&\rank{\begin{bmatrix}
zX_p-X_f \\ U_p \\ Y_p\\
Y_f 
\end{bmatrix}} \!\!= \!
\rank{\left(
\begin{bmatrix}
 -B & -E  & zI-A \\
I & 0 & 0 \\
0 & 0 & C\\
CB & CE & CA
\end{bmatrix}\!\!
\begin{bmatrix}
U_p \\
D_p\\
X_p
\end{bmatrix}
\!\right)} \\
\!\!\!\!\!\!\!\!\!\!\!\!&&=\rank{\left(
\begin{bmatrix}
 I & 0 & 0 & 0 \\
0 & I & 0 & 0 \\
0 & 0 & I & 0 \\
C & 0 & -zI & I
\end{bmatrix}\!\!
\begin{bmatrix}
 -B & -E  & zI-A \\
I & 0 & 0 \\
0 & 0 & C\\
CB & CE & CA
\end{bmatrix}\!\!
\begin{bmatrix}
U_p \\
D_p\\
X_p
\end{bmatrix}\!
\right)} \\
\!\!\!\!\!\!\!\!\!\!\!\!&&= 
\rank{\left(\!\begin{bmatrix}
 -B & -E  & zI-A \\
I & 0 & 0 \\
  0 & 0 & C\\
0 & 0 & 0 \\
\end{bmatrix}\!\!\!
\begin{bmatrix}
U_p \\
D_p\\
X_p
\end{bmatrix}\!\!
\right)} \!\!=\! \rank{\!\!\begin{bmatrix}
zX_p-X_f \\ U_p \\ Y_p
\end{bmatrix}}.
\end{eqnarray*}
On the other hand, by exploiting condition {\em (ii)}, we obtain
$$\begin{bmatrix}
zX_p-X_f \\
U_p \\
Y_p \\
Y_f
\end{bmatrix} = \begin{bmatrix}
- T_1 & -T_2 & -T_3 & zI - T_4\cr
I & 0 & 0 & 0\cr
0 & I & 0 & 0\cr
0 & 0 & I &0\end{bmatrix}
\begin{bmatrix}
U_p \\
Y_p \\
Y_f \\
X_p 
\end{bmatrix}.
$$
Since $T_4$ is Schur, it follows that the rank of the matrix on the left coincides with 
$\rank{\begin{bmatrix}
U_p^\top &
Y_p^\top &
Y_f^\top &
X_p^\top 
\end{bmatrix}^\top}$
 for every $z \in \mathbb C$ with $|z|\ge 1$. Finally,
 $$\rank{\begin{bmatrix}
U_p \\
Y_p \\
Y_f \\
X_p 
\end{bmatrix}} \!\!=\rank\left(\!
\begin{bmatrix} I &0 &0\cr 0 & 0 &C\cr
CB & CE & CA\cr
0 & 0 & I\end{bmatrix} \!\!\begin{bmatrix} U_p\cr D_p\cr X_p\end{bmatrix}\!\!\right) \!\!=\! n+m+r,$$
where we exploited the Assumption and the fact that $\rank{(CE)} = \rank{(E)} = r$. 
 
\noindent
\emph{(iii)} $\Rightarrow$ \emph{(iv)}. In Proposition \ref{kernelincl} we proved that condition \eqref{ker_inclusion} is equivalent to
condition {\em (iv)}, point {\em (a)}, of Theorem \ref{cns_MB}, namely to 
$\rank (CE)=\rank (E)=r$. \\
On the other hand, it is easy to see that
$$\begin{bmatrix}
z X_p - X_f\cr U_p \cr Y_p\end{bmatrix} \!\!=\!\! \begin{bmatrix}
- B & - E & z I_n -A\cr I_m &0 &0\cr 0 & 0 &C\end{bmatrix}\begin{bmatrix}
U_p \\ D_p \\ X_p
\end{bmatrix},$$
and, as result of the Assumption, for every $z\in {\mathbb C}$,
$$\rank\!\left(\begin{bmatrix}
z X_p - X_f\cr U_p \cr Y_p\end{bmatrix}\right) \!\!=\! \rank\left(\!\begin{bmatrix}
- B & - E & z I_n -A\cr I_m &0 &0\cr 0 & 0 &C\end{bmatrix}\!\right)$$
$$=
m +  \rank\left(\begin{bmatrix}
 - E & z I_n -A\cr  0 &C\end{bmatrix}\right).$$
 Consequently, condition \eqref{dd_rank} is equivalent to condition {\em (iv)}, point {\em (b)}, of Theorem \ref{cns_MB}.
Thus, by Theorem \ref{cns_MB} the system 
$\Sigma$ is strong* detectable.

\noindent
\emph{(iv)} $\Rightarrow$ \emph{(i)}. If {\em (iv)} holds, by Theorem \ref{cns_MB} we know that there exists a UIO $\hat \Sigma$ for $\Sigma$ described as
in 
\eqref{UIO_eq1}-\eqref{UIO_eq2}. 
By the proof of \emph{(iv)} $\Rightarrow$ \emph{(i)}
in Proposition \ref{kernelincl}, we can claim 
 that $T_{\Sigma} \subseteq T_{\hat \Sigma}$.
\end{proof}

{\color{black} The previous theorem  gives a complete answer to the question of whether it is possible to design a UIO based only on some available data. 
Indeed, condition {\em (iii)} provides a way to check a priori on the collected data if a data-driven UIO exists. In addition, the same condition is shown to be equivalent to condition {\em (iv)}, meaning that, under the Assumption on the data, solving the problem via a data-driven approach does not introduce additional constraints with respect to those obtained in the model-based formulation. 
Furthermore, once the UIO existence  has been ascertained, we can exploit the fact of having introduced condition {\em (ii)} in Proposition \ref{kernelincl} to relate the solutions of equation \eqref{Xf} to all the possible quadruples $(A_{UIO}, B_{UIO}^u, B_{UIO}^y, D_{UIO})$ satisfying \eqref{SS_cond1}$\div$\eqref{SS_cond4}. This will be the subject of the following corollary.}

\begin{corollary}\label{param}
Under the Assumption on the data, if any of the equivalent conditions of Theorem \ref{UIO_final} holds, then
 there is a bijective correspondence between the matrices 
$(A_{UIO}, B_{UIO}^u,B_{UIO}^y, D_{UIO})$ describing a  UIO $\hat \Sigma$  and the matrices $T_1, T_2, T_3,$ and $T_4$, with $T_4$ Schur, such 
that \eqref{Xf} holds. \end{corollary}

\begin{proof} We will prove that  there is a bijective correspondence between the matrices 
$(A_{UIO}, B_{UIO}^u,B_{UIO}^y, D_{UIO})$ describing a  system $\hat \Sigma$ for which $\mathbb{T}_\Sigma \subseteq\mathbb{T}_{\hat{\Sigma}}$ and the matrices $T_1, T_2, T_3,$ and $T_4$ such 
that \eqref{Xf} holds.
Since it is always true that $T_4=A_{UIO}$, the corollary statement immediately follows. \smallskip

 From the proof of {\em (i)} $\Rightarrow$ {\em (ii)} in Proposition \ref{kernelincl}, we have seen that  every 
quadruple of matrices $(A_{UIO}, B_{UIO}^u, B_{UIO}^y, D_{UIO})$  describing a  system   \eqref{UIO_eq1}-\eqref{UIO_eq2} such that $\mathbb{T}_\Sigma \subseteq\mathbb{T}_{\hat{\Sigma}}$ 
identifies (through \eqref{daUIOaT1} and \eqref{daUIOaT2})
a quadruple $(T_1, T_2, T_3,T_4)$  such 
that \eqref{Xf} holds.\\
Conversely,  suppose that  
$\exists 
 (T_1, T_2, T_3,T_4)$ s.t. \eqref{Xf} holds.
 Now, set 
\begin{eqnarray}
A_{UIO}&\triangleq& T_4 , \label{Auio} \\ 
B_{UIO}^u &\triangleq& T_1,  \ \
B_{UIO}^y  \triangleq T_2+T_4T_3, \ \
D_{UIO} \triangleq T_3, \qquad \label{Duio}
\end{eqnarray}
yielding 
$$
X_f \!=\!  \left[\!\!\!\begin{array}{c|c|c|c}
B_{UIO}^u \!&\! B_{UIO}^y-A_{UIO}D_{UIO} \!&\! D_{UIO} \!&\! A_{UIO}
 \end{array}
 \!\!\!\right] \!\!
\begin{bmatrix}
U_p \\
Y_p \\
Y_f \\
X_p 
\end{bmatrix}.
$$
We want to prove that such matrices describe a system \eqref{UIO_eq1}-\eqref{UIO_eq2} such that $\mathbb{T}_\Sigma \subseteq\mathbb{T}_{\hat{\Sigma}}$ holds. To this end we preliminarily note that
if $(\{u(t)\}_{t\in {\mathbb Z}_+},\{y(t)\}_{t\in {\mathbb Z}_+},$ $\{x(t)\}_{t\in {\mathbb Z}_+})\in\mathbb{T}_c(u_d,y_d,x_d)$, then
for every $t\in {\mathbb Z}_+$ there exists $g_t\in {\mathbb R}^{T-1}$ such that  \eqref{gt} holds.
This implies that $y(t) = Y_p g_t = C X_p g_t = C x(t)$ for every $t\in {\mathbb Z}_+$, and hence
$y(t+1)= C x(t+1)= CX_f g_t = Y_f g_t$ for every $t\in {\mathbb Z}_+$.
 Therefore, for every trajectory $(\{u(t)\}_{t\in {\mathbb Z}_+},\{y(t)\}_{t\in {\mathbb Z}_+},$ $\{x(t)\}_{t\in {\mathbb Z}_+})\in\mathbb{T}_c(u_d,y_d,x_d)$, it holds 
\begin{eqnarray*}
x(t+1) \!\!\!\!\!&=&\!\!\!\!\! A_{UIO}x(t)\!+\!B_{UIO}^uu(t)\!+\!(B_{UIO}^y-A_{UIO}D_{UIO})y(t) \\
\!\!\!\!\!&+&\!\!\!\!\!D_{UIO}y(t+1).
\end{eqnarray*}
Define
$z(t) \triangleq x(t) - D_{UIO}y(t).$
Then the trajectory $(\{u(t)\}_{t\in {\mathbb Z}_+},\{y(t)\}_{t\in {\mathbb Z}_+},\{x(t)\}_{t\in {\mathbb Z}_+},\{z(t)\}_{t\in {\mathbb Z}_+})$ satisfies 
\begin{eqnarray*}
x(t) \!\!\!&=&\!\!\! z(t) + D_{UIO}y(t) \\
z(t+1) \!\!\!&=&\!\!\! x(t+1) - D_{UIO}y(t+1) \\
\!\!\!&=&\!\!\! A_{UIO}x(t) + B_{UIO}^u u(t) \\
\!\!\!&+&\!\!\! (B_{UIO}^y-A_{UIO}D_{UIO})y(t) \\
\!\!\!&=&\!\!\! A_{UIO}z(t)+B_{UIO}^u u(t) + B_{UIO}^y y(t).
\end{eqnarray*}
This proves that $\forall (\{u(t)\}_{t\in {\mathbb Z}_+},\{y(t)\}_{t\in {\mathbb Z}_+},\{x(t)\}_{t\in {\mathbb Z}_+})\in\mathbb{T}_c(u_d,y_d,x_d)$ there exists $\{z(t)\}_{t\in {\mathbb Z}_+}$ such that $(\{u(t)\}_{t\in {\mathbb Z}_+},\{y(t)\}_{t\in {\mathbb Z}_+},\{x(t)\}_{t\in {\mathbb Z}_+},\{z(t)\}_{t\in {\mathbb Z}_+})$ satisfies equations \eqref{UIO_eq1}-\eqref{UIO_eq2}, and hence 
$(A_{UIO}, B_{UIO}^u, B_{UIO}^y, D_{UIO})$  describes a  system   \eqref{UIO_eq1}-\eqref{UIO_eq2} for which $\mathbb{T}_\Sigma \subseteq\mathbb{T}_{\hat{\Sigma}}$.
\end{proof} 
}

\section{{\color{black} A simplified way to compute the problem solution}} \label{concl_rem}

{\color{black}  The two conditions given in {\em (iii)} of Theorem \ref{UIO_final} 
provide a practical way to check on data whether the problem of designing a UIO is solvable. However, 
such conditions do not lead to an explicit solution, namely to a quadruple of matrices 
 $(T_1,T_2,T_3, T_4)$,  with $T_4$   Schur stable, such that \eqref{Xf} holds.
 To this end we can replace equation \eqref{Xf} with a simpler equivalent equation.}
\\
We can observe that $Y_p = CX_p$  since the data have been generated by 
$\Sigma$. 
As a consequence of the Assumption, the matrix $X_p$ is of full row rank. This implies that $C$ can be uniquely recovered from the data as
$$C = Y_p X_p^\dag = Y_p X_p^\top (X_p X_p^\top)^{-1}.$$
Moreover, equation \eqref{Xf} is equivalent to 
\begin{eqnarray}
X_f &=& \left[\begin{array}{c|c|c|c}
 T_1 & T_2 & T_3 & T_4
 \end{array}
 \right]
 \begin{bmatrix}
 I & 0 & 0 \\
 0 & 0 & C \\
 0 & I & 0 \\
 0 & 0 & I
 \end{bmatrix}
\begin{bmatrix}
U_p \\
Y_f \\
X_p 
\end{bmatrix} \nonumber\\
&=& \left[\begin{array}{c|c|c} 
 T_1 & T_3 & T_4 + T_2 C
 \end{array}
 \right]
 \begin{bmatrix}
U_p \\
Y_f \\
X_p
\end{bmatrix}. \label{detectable_pair}
\end{eqnarray}
Therefore, $A_{UIO}=T_4$ can be a Schur stable matrix if and only if there exists a triple $(T_1,T_3, T^*)$ that solves the equation:
 \be
 \label{Xfshort}
 X_f = \left[\begin{array}{c|c|c} 
 T_1 & T_3 & T^*
 \end{array}
 \right]
  \begin{bmatrix}
U_p \\
Y_f \\
X_p
\end{bmatrix}
\ee
such that the pair $(T^*,C)$  is detectable in the sense of discrete-time linear systems. Indeed,  this amounts to saying that $\exists T_2$ such that $T^*-T_2C=T_4=A_{UIO}$ is Schur stable. 
\\
{\color{black}In a recent paper \cite{FDI_Giulio_arXiv}, an algorithm to explicitly determine (if it exists) a triple $(T_1, T_3, T^*)$, with $(T^*,C)$ detectable, 
that solves \eqref{Xfshort} is proposed. This provides a practical way to design a UIO from data, under the assumption that the two conditions
given in {\em (iii)} of Theorem \ref{UIO_final}  hold.
We refer the interested reader to 
\cite{FDI_Giulio_arXiv}.\\
To conclude the paper, we provide a numerical example that illustrates how it is possible  to design a UIO
both from a model-based perspective and from data.
}
 
 {\color{black} 
 \begin{example}
 Assume (as in \cite{Darouach}, Example 2)
$$A=\begin{bmatrix} -1 & -1 & 0\cr -1& 0 & 0\cr 0 & -1 & -1\end{bmatrix}, \quad 
C = \begin{bmatrix} 1 & 0 & 0\cr 0 & 0 &1\end{bmatrix}, \quad
E = \begin{bmatrix}
-1 \cr 0 \cr 0\end{bmatrix}.$$
The matrix $B$ is omitted since the presence of the known input $u$ can be easily handled without requiring additional design steps.  \\
We first solve the problem by adopting a model-based approach.
It is a matter of elementary calculations to verify that conditions  (a) and   (b) in point  (iii) of Theorem \ref{cns_MB} hold, and hence a UIO exists.
In fact, 
$$
{\rm rank} \begin{bmatrix} zI_n - A & -E\cr
C & 0\end{bmatrix} = n+r, \ \forall z\in {\mathbb C},$$
and hence the triple $(A,E,C)$ is  not only strong* detectable, but also strong* observable.
 The set of matrices 
$D_{UIO}$ such that \eqref{SS_cond2} holds can be described as
\be
D_{UIO} = \begin{bmatrix} 1 & a\cr 0 &b \cr 0 &c\end{bmatrix}, \quad a,b,c \in {\mathbb R}.
\label{paramDUIO}
\ee
In order to find matrices $A_{UIO}$ and $B_{UIO}^y$ such that \eqref{SS_cond1} and  \eqref{SS_cond4} hold, we can rewrite 
\eqref{SS_cond4} as
$
A_{UIO}= (I-D_{UIO}C)A - L C,$
where $L = B_{UIO}^y- A_{UIO} D_{UIO}$.
So, we can first choose $D_{UIO}$ as in \eqref{paramDUIO} in such a way that 
$((I-D_{UIO}C)A, C)$ is either observable or at least detectable.
Then we can choose $L$ so that $A_{UIO}= (I-D_{UIO}C)A - L C$ is Schur, and finally determine $B_{UIO}^y$ from $L$.
It turns out that
$$((I-D_{UIO}C)A, C) = \left(\begin{bmatrix} 0 & a & a\cr -1 & b & b\cr
0 & c-1 & c-1\end{bmatrix}, \begin{bmatrix} 1 & 0 & 0\cr 0 & 0 &1\end{bmatrix}\right)$$
is observable if and only if either $a\ne 0$ or $c\ne 1$. While if $a=0$ and $c=1$, then the pair is detectable if and only if $|b|< 1$. Therefore, we can essentially distinguish between two cases: 
\begin{itemize}
\item[1)] $a\ne 0 \vee c\ne 1 \vee |b|<1$, for which the pair $((I-D_{UIO}C)A, C)$ is  at least detectable;
\item[2)] $a=0 \wedge c =1 \wedge |b|\ge1$, for which it is not. 
\end{itemize}
A possible solution  (corresponding to $a=b =0$, $c=1$) is 
$$D_{UIO} = \begin{bmatrix} 1 & 0\cr 0 &0 \cr 0 &1\end{bmatrix},$$
 that makes $((I-D_{UIO}C)A, C)$  reconstructable.
 Moreover,
 $$ (I-D_{UIO}C)A = \begin{bmatrix} 0 & 0 & 0 \cr -1 & 0 & 0\cr
0 & 0 & 0 \end{bmatrix}.$$
So, we can simply choose $L=0$ and $A_{UIO}= (I-D_{UIO}C)A$
 is nilpotent.

We now tackle the problem using the proposed data driven procedure. We set $T=20$. We generate the historical unknown input data by varying each component randomly and uniformly in the interval $(-2,2)$, so that it is reasonable to think that the Assumption is satisfied. We collect the corresponding output and state data in the time interval $[0,T-1]$. We use the collected data to reconstruct the matrix $C$ and to verify that 
the two conditions in Theorem \ref{UIO_final}, point (iii), hold. Then, we compute the following particular solution to equation \eqref{Xfshort} (note that there is no $U_p$), namely 
$$
\left[\begin{array}{c|c}
T_3 & T^*
\end{array}\right] = X_f \begin{bmatrix}
Y_f \\
X_p
\end{bmatrix}^\dag,
$$
and we verify that  the pair $(T^*,C)$ is detectable.  We select $T_2$ in order to make the matrix $A_{UIO} = T_4=  T^*-T_2C$ nilpotent, namely  
$$
T_2 = \begin{bmatrix}
    0  & 0 \\
   -1 & 0 \\
    0 &  - 1/3
\end{bmatrix}  \ \Rightarrow \
A_{UIO} = 
 \begin{bmatrix}
    0  &  0  & 0 \\
   0  &  0 &  0 \\
    0  & -1/3 &  0
\end{bmatrix}. 
$$
Finally, the matrices $B_{UIO}^y$ and $D_{UIO}$ are obtained from $(T_2, T_3, T_4)$  using  \eqref{Duio}.
With the approach proposed in \cite{Ferrari-Trecate} we would have obtained 
$$
A_{UIO} = 
 \begin{bmatrix}
   0  & 0  & 0 \\
   -0.5 &  0  &  0 \\
   0  & -0.4 &  -0.2
\end{bmatrix}, 
$$
whose eigenvalues are $\{0,0,-0.2\}$. Therefore, in this case both  
our solution and the one proposed in \cite{Ferrari-Trecate} work, but our approach has the advantage of allowing us to choose the observable eigenvalues of the matrix $A_{UIO}$ and, consequently,  the error convergence speed. The dynamics of the state estimation error in the two cases is illustrated in Figure \ref{figure1}, corresponding to a random initial condition and  a random disturbance taking values in $(-10,10)$. The solid black line is related to our design procedure, while the dashed red line corresponds to the solution in \cite{Ferrari-Trecate}. 
\begin{figure}
\centering
\! \includegraphics[width = 0.49\textwidth]{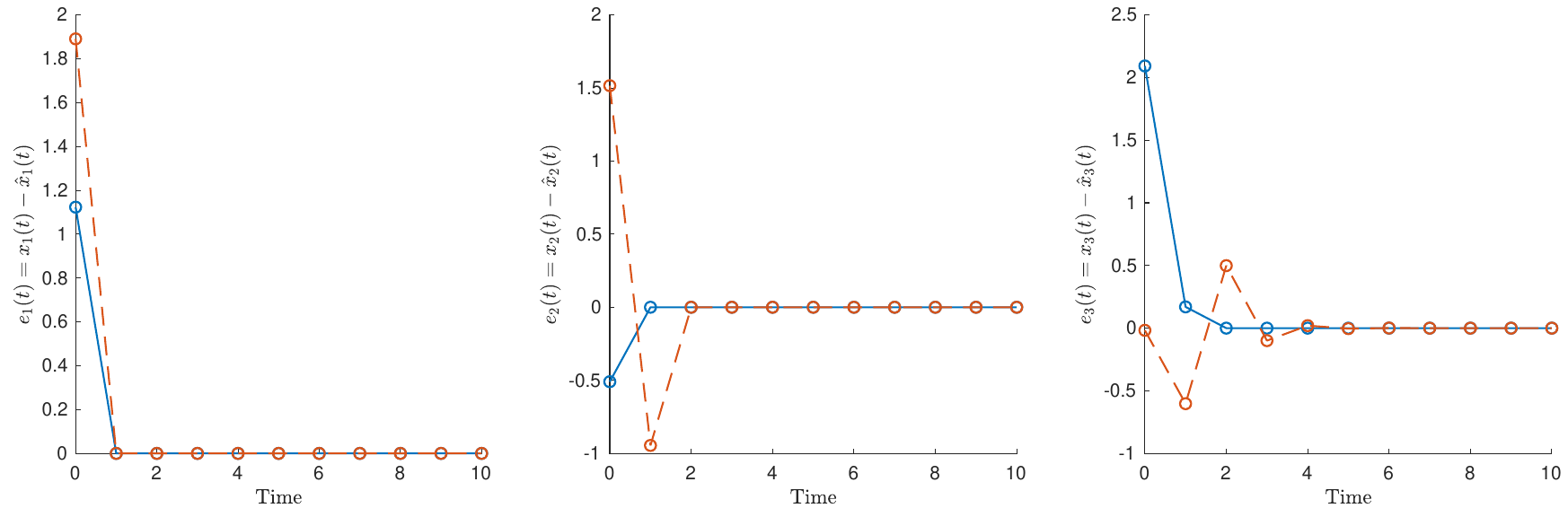}
\caption{Dynamics of the state estimation error component-wise}
 \label{figure1}
\end{figure} 
\end{example}}

\section{Conclusions}\label{concl}
 In this paper, we first revised the solution to the  
  UIO design problem from the model-based perspective. Then, we 
   provided necessary and sufficient conditions for the problem solvability via a  data-driven approach. 
 If the collected data are representative of the system dynamics,  the solvability conditions derived in the data-driven setting are equivalent to the    classic model-based ones {\color{black} and they can be tested a priori on data only.

 Our results represent   an improvement of the results that can be found in \cite{Ferrari-Trecate} due to the following contributions:  (1) the existence of a 
necessary and sufficient condition for the existence of a UIO that can be checked a priori  on the data,
  (2) the proof of
the equivalence of the model-based and data-driven approaches to the problem solution (under the Assumption on the data), and
(3)  the existence of  a bijective correspondence between the matrices 
$(A_{UIO}, B_{UIO}^u,B_{UIO}^y, D_{UIO})$ describing a  UIO $\hat \Sigma$  (see \eqref{SS_cond1}$\div$\eqref{SS_cond4}) and the matrices $T_1, T_2, T_3,$ and $T_4$, with $T_4$ Schur, such 
that \eqref{Xf} holds.    }

\bibliographystyle{plain}

\bibliography{
Refer175}

\end{document}